\documentclass{article}
\usepackage{amssymb}

\newtheorem{theorem}{Theorem}[section]

\newtheorem{corollary}[theorem]{Corollary}

\newtheorem{definition}[theorem]{Definition}
\newtheorem{example}[theorem]{Example}

\newtheorem{lemma}[theorem]{Lemma}

\newtheorem{proposition}[theorem]{Proposition}

\newenvironment{proof}[1][Proof]{\noindent\textbf{#1.} }{\ \rule{0.5em}{0.5em}}
\input{tcilatex}

\begin{document}

\title{\bigskip G-RADICAL SUPPLEMENTED\ MODULES}
\author{Celil Nebiyev \\
\\
$Department$ $of$ $Mathematics,Ondokuz$ $May\imath s$ $University,$\\
$55270$ $Kurupelit-Atakum$/$Samsun$/$T\ddot{U}RK\dot{I}YE-TURKEY$\ \ \\
$cnebiyev@omu.edu.tr$}
\date{ \qquad }
\maketitle

\begin{abstract}
In this work, g-radical supplemented modules which is a proper
generalization of g-supplemented modules are defined and some properties of
these modules are investigated. It is proved that the finite sum of
g-radical supplemented modules is g-radical supplemented. It is also proved
that every factor module and every homomorphic image of a g-radical
supplemented module is g-radical supplemented. Let $R$ be a ring. Then $%
_{R}R $ is g-radical supplemented if and only if every finitely generated $R$%
-module is g-radical supplemented. In the end of this work, it is given two
examples for g-radical supplemented modules seperating with g-supplemented
modules.
\end{abstract}

\date{\textbf{Key words:} Small Submodules, Radical, Supplemented Modules,
Radical (Generalized) Supplemented Modules.}

\section{INTRODUCTION}

In this paper, all rings are associative with identity and all modules are
unital left modules.

Let $M$ be an $R$ -module and $N$ $\leq M$. If $L=M$ for every submodule $L$
of $M$ such that $M=N+L$, then $N$ is called a small submodule of $M$,
denoted by $N<<M$. Let $M$ be an $R$ -module and $N\leq M$. If there exists
a submodule $K$ of $M$ such that $M=N+K$ and $N\cap K=0$, then $N$ is called
a direct summand of $M$ and it is denoted by $M=N\oplus K$. For any module $%
M,\ $we have $M=M\oplus 0$. The intersection of all maximal submodules of an 
$R$-module $M$ is called the radical of $M$ and denoted by $RadM$. If $M$
have no maximal submodules, then we call $RadM=M$. A submodule $N$ of an $R$
-module $M$ is called an essential submodule and denoted by $%
N\trianglelefteq M$ in case $K\cap N\neq 0$ for every submodule $K\neq 0$.
Let $M$ be an $R$ -module and $K$ be a submodule of $M$. $K$ is called a
generalized small submodule of $M$ denoted by $K<<_{g}M$ if for every
essential submodule $T$ of $M$ with the property $M=K+T$ implies that $T=M$.
It is clear that every small submodule is a generalized small submodule but
the converse is not true generally. Let $M$ be an $R-$module. $M$ is called
an hollow module if every proper submodule of $M$ is small in $M$. $M$ is
called local module if $M$ has a largest submodule, i.e. a proper submodule
which contains all other proper submodules. Let $U$ and $V$ be submodules of 
$M$. If $M=U+V$ and $V$ is minimal with respect to this property, or
equivalently, $M=U+V$ and $U\cap V<<V$, then $V$ is called a supplement of $%
U $ in $M$. $M$ is called a supplemented module if every submodule of $M$
has a supplement in $M$. Let $M$ be an $R$-module and $U,V\leq M$. If $M=U+V$
and $M=U+T$ with $T\trianglelefteq V$ implies that $T=V$, or equivalently, $%
M=U+V$ and $U\cap V\ll _{g}M$, then $V$ is called a g-supplement of $U$ in $%
M $. $M$ is called g-supplemented if every submodule of $M$ has a
g-supplement in $M$. The intersection of maximal essential submodules of an $%
R$-module $M$ is called a generalized radical of $M$ and denoted by $%
Rad_{g}M $. If $M$ have no maximal essential submodules, then we denote $%
Rad_{g}M=M.$

\begin{lemma}
\label{1}Let $M$ be an $R$ -module and $K,L,N,T\leq M$. Then the followings
are hold. $\left[ \text{\ref{1.3},\ref{1.4}}\right] $

$\left( 1\right) $ If $K\leq N$ and $N$ is generalized small submodule of $M$%
, then $K$ is a generalized small submodule of $M$.

$\left( 2\right) $ If $K$ is contained in $N$ and a generalized small
submodule of $N$, then $K$ is a generalized small submodule in submodules of 
$M$ which contains submodule $N$.

$\left( 3\right) $ Let $f:M\rightarrow N$ be an $R$ -module homomorphism. If 
$K\ll _{g}M$ , then $f\left( K\right) \ll _{g}M$.

$\left( 4\right) $ If $K\ll _{g}L$ and $N\ll _{g}T$, then $K+N\ll _{g}L+T$.
\end{lemma}

\begin{corollary}
Let $M_{1},M_{2},...,M_{n}\leq M$, $K_{1}\ll _{g}M_{1}$, $K_{2}\ll _{g}M_{2}$%
, ..., $K_{n}\ll _{g}M_{n}$. Then $K_{1}+K_{2}+...+K_{n}\ll
_{g}M_{1}+M_{2}+...+M_{n}$.
\end{corollary}

\begin{corollary}
Let $M$ be an $R$ -module and $K\leq N\leq M$ . If $N\ll _{g}M$ , then $%
N/K\ll _{g}M/K$.
\end{corollary}

\begin{corollary}
Let $M$ be an $R$ -module, $K\ll _{g}M$ and $L\leq M$. Then $\left(
K+L\right) /L\ll _{g}M/L$.
\end{corollary}

\begin{lemma}
\label{2}Let $M$ be an $R$-module. Then $Rad_{g}M=\sum_{L\ll _{g}M}L.$
\end{lemma}

\begin{proof}
See$\left[ \text{\ref{1.3}}\right] $.
\end{proof}

\begin{lemma}
\label{3}The following assertions are hold.

$\left( 1\right) $ If $M$ is an $R-$module, then $Rm\ll _{g}M$ for every $%
m\in Rad_{g}M$.

$\left( 2\right) $ If $N\leq M$, then $Rad_{g}N\leq Rad_{g}M.$

$\left( 3\right) $ If $K,L\leq M$, then $Rad_{g}K+Rad_{g}L\leq Rad_{g}\left(
K+L\right) .$

$\left( 4\right) $ If $f:M\longrightarrow N$ is an $R$-module homomorphism,
then $f\left( Rad_{g}M\right) \leq Rad_{g}N.$

$\left( 5\right) $ If $K,L\leq M$, then $Rad_{g}\frac{K+L}{L}\leq \frac{%
Rad_{g}K+L}{L}.$
\end{lemma}

\begin{proof}
Clear from Lemma \ref{1} and Lemma \ref{2}.
\end{proof}

\begin{lemma}
\label{8}Let $M=\oplus _{i\in I}M_{i}.$ Then $Rad_{g}M=\oplus _{i\in
I}Rad_{g}M_{i}.$
\end{lemma}

\begin{proof}
Since $M_{i}\leq M$, then by Lemma \ref{3}$\left( 2\right) $, $%
Rad_{g}M_{i}\leq Rad_{g}M$ and $\oplus _{i\in I}Rad_{g}M_{i}\leq Rad_{g}M.$
Let $x\in Rad_{g}M.$ Then by Lemma \ref{3}$\left( 1\right) $, $Rx\ll _{g}M.$
Since $x\in M=\oplus _{i\in I}M_{i}$, there exist $i_{1},i_{2},...,i_{k}\in
I $ and $x_{i_{1}}\in M_{i_{1}}$, $x_{i_{2}}\in M_{i_{2}}$, ..., $%
x_{i_{k}}\in M_{i_{k}}$ such that $x=x_{i_{1}}+x_{i_{2}}+...+x_{i_{k}}$.
Since $Rx\ll _{g}M$ , then by Lemma \ref{1}$\left( 4\right) $, under the
canonical epimorphism $\pi _{i_{t}}$ $\left( t=1,2,...,k\right) $ $%
Rx_{i_{t}}=\pi _{i_{t}}\left( Rx\right) \ll _{g}Rx_{i_{t}}.$ Then $%
x_{i_{t}}\in Rad_{g}M_{i_{t}}$ $\left( t=1,2,...,k\right) $ and $%
x=x_{i_{1}}+x_{i_{2}}+...+x_{i_{k}}\in \oplus _{i\in I}Rad_{g}M_{i}$. Hence $%
Rad_{g}M\leq \oplus _{i\in I}Rad_{g}M_{i}$ and since $\oplus _{i\in
I}Rad_{g}M_{i}\leq Rad_{g}M$, $Rad_{g}M=\oplus _{i\in I}Rad_{g}M_{i}$.
\end{proof}

\section{G-RADICAL\ SUPPLEMENTED\ MODULES}

\begin{definition}
Let $M$ be an $R$-module and $U,V\leq M$. If $M=U+V$ and $U\cap V\leq
Rad_{g}V$, then $V$ is called a generalized radical supplement $\left( \text{%
briefly, g-radical supplement}\right) $ of $U$ in $M$. If every submodule of 
$M$ has a generalized radical supplement in $M$, then $M$ is called a
generalized radical supplemented $\left( \text{briefly, g-radical
supplemented}\right) $ module.
\end{definition}

Clearly we see that every g-supplemented module is g-radical supplemented.
But the converse is not true in general $\left( \text{See Example \ref{14}
and Example \ref{15}}\right) $.

\begin{lemma}
\label{13}Let $M$ be an $R$-module and $U,V\leq M$. Then $V$ is a g-radical
supplement of $U$ in $M$ if and only if $M=U+V$ and $Rm\ll _{g}V$ for every $%
m\in U\cap V$.
\end{lemma}

\begin{proof}
$\left( \Rightarrow \right) $ Since $V$ is a g-radical supplement of $U$ in $%
M$, $M=U+V$ and $U\cap V\leq Rad_{g}V$. Let $m\in U\cap V$. Since $U\cap
V\leq Rad_{g}V$, $m\in Rad_{g}V$. Hence by Lemma \ref{3}$\left( 1\right) $, $%
Rm\ll _{g}V$.

$\left( \Leftarrow \right) $ Since $Rm\ll _{g}V$ for every $m\in U\cap V$,
then by Lemma \ref{3}$\left( 1\right) $, $U\cap V\leq Rad_{g}V$ and hence $V$
is a g-radical supplement of $U$ in $M$.
\end{proof}

\begin{lemma}
\label{4}Let $M$ be an $R$-module, $M_{1},U,X\leq M$ and $Y\leq M_{1}$. If $%
X $ is a g-radical supplement of $M_{1}+U$ in $M$ and $Y$ is a g-radical
supplement of $\left( U+X\right) \cap M_{1}$ in $M_{1}$, then $X+Y$ is a
g-radical supplement of $U$ in $M$.
\end{lemma}

\begin{proof}
Since $X$ is a g-radical supplement of $M_{1}+U$ in $M$, $M=M_{1}+U+X$ and $%
\left( M_{1}+U\right) \cap X\leq Rad_{g}X.$ Since $Y$ is a g-radical
supplement of $\left( U+X\right) \cap M_{1}$ in $M_{1}$, $M_{1}=\left(
U+X\right) \cap M_{1}+Y$ and $\left( U+X\right) \cap Y=\left( U+X\right)
\cap M_{1}\cap Y\leq Rad_{g}Y$. Then $M=M_{1}+U+X=M_{1}=\left( U+X\right)
\cap M_{1}+Y+U+X=U+X+Y$ and, by Lemma \ref{3}$\left( 3\right) $, $U\cap
\left( X+Y\right) \leq \left( U+X\right) \cap Y+\left( U+Y\right) \cap X\leq
Rad_{g}Y+\left( M_{1}+U\right) \cap X\leq Rad_{g}Y+Rad_{g}X\leq
Rad_{g}\left( X+Y\right) $. Hence $X+Y$ is a g-radical supplement of $U$ in $%
M$.
\end{proof}

\begin{lemma}
\label{5}Let $M=M_{1}+M_{2}$. If $M_{1}$ and $M_{2}$ are g-radical
supplemented, then $M$ is also g-radical supplemented.
\end{lemma}

\begin{proof}
Let $U\leq M$. Then $0$ is a g-radical supplement of $M_{1}+M_{2}+U$ in $M$.
Since $M_{1}$ is g-radical supplemented, there exists a g-radical supplement 
$X$ of $\left( M_{2}+U\right) \cap M_{1}=\left( M_{2}+U+0\right) \cap M_{1}$
in $M_{1}$. Then by Lemma \ref{4}, $X+0=X$ is a g-radical supplement of $%
M_{2}+U$ in $M$. Since $M_{2}$ is g-radical supplemented, there exists a
g-radical supplement $Y$ of $\left( U+X\right) \cap M_{2}$ in $M_{2}$. Then
by Lemma \ref{4}, $X+Y$ is a g-radical supplement of $U$ in $M$.
\end{proof}

\begin{corollary}
\label{10}Let $M=M_{1}+M_{2}+...+M_{k}$. If $M_{i}$ is g-radical
supplemented for every $i=1,2,...,k$, then $M$ is also g-radical
supplemented.
\end{corollary}

\begin{proof}
Clear from Lemma \ref{5}.
\end{proof}

\begin{lemma}
\label{6}Let $M$ be an $R-$module, $U,V\leq M$ and $K\leq U$. If $V$ is a
g-radical supplement of $U$ in $M$, then $\left( V+K\right) /K$ is a
g-radical supplement of $U/K$ in $M/K$.
\end{lemma}

\begin{proof}
Since $V$ is a g-radical supplement of $U$ in $M$, $M=U+V$ and $U\cap V\leq
Rad_{g}V$. Then $M/K=U/K+\left( V+K\right) /K$ and by Lemma \ref{3}$\left(
5\right) $, $\left( U/K\right) \cap \left( \left( V+K\right) /K\right)
=\left( U\cap V+K\right) /K\leq \left( Rad_{g}V+K\right) /K\leq
Rad_{g}\left( V+K\right) /K$. Hence $\left( V+K\right) /K$ is a g-radical
supplement of $U/K$ in $M/K$.
\end{proof}

\begin{lemma}
\label{7}Every factor module of a g-radical supplemented module is g-radical
supplemented.
\end{lemma}

\begin{proof}
Clear from Lemma \ref{6}.
\end{proof}

\begin{corollary}
\label{11}The homomorphic image of a g-radical supplemented module is
g-radical supplemented.
\end{corollary}

\begin{proof}
Clear from Lemma \ref{7}.
\end{proof}

\begin{lemma}
\label{12}Let $M$ be a g-radical supplemented module. Then every finitely $%
M- $generated module is g-radical supplemented.
\end{lemma}

\begin{proof}
Clear from Corollary \ref{10} and Corollary \ref{11}.
\end{proof}

\begin{corollary}
Let $R$ be a ring. Then $_{R}R$ is g-radical supplemented if and only if
every finitely generated $R-$module is g-radical supplemented.
\end{corollary}

\begin{proof}
Clear from Lemma \ref{12}.
\end{proof}

\begin{theorem}
Let $M$ be an $R-$module. If $\ M$ is g-radical supplemented, then $%
M/Rad_{g}M$ is semisimple.
\end{theorem}

\begin{proof}
Let $U/Rad_{g}M\leq M/Rad_{g}M$. Since $M$ is g-radical supplemented, there
exists a g-radical supplement $V$ of $U$ in $\ M$. Then $M=U+V$ and $U\cap
V\leq Rad_{g}V$. Thus $M/Rad_{g}M=U/Rad_{g}M+\left( V+Rad_{g}M\right)
/Rad_{g}M$ and $\left( U/Rad_{g}M\right) \cap \left( \left(
V+Rad_{g}M\right) /Rad_{g}M\right) =\left( U\cap V+Rad_{g}M\right)
/Rad_{g}M\leq \left( Rad_{g}V+Rad_{g}M\right) /Rad_{g}M=Rad_{g}M/Rad_{g}M=0$%
. Hence $M/Rad_{g}M=U/Rad_{g}M\oplus \left( V+Rad_{g}M\right) /Rad_{g}M$ and 
$U/Rad_{g}M$ is a direct summand of $M$.
\end{proof}

\begin{lemma}
\label{9}Let $M$ be a g-radical supplemented module and $L\leq M$ with $%
L\cap Rad_{g}M=0$. Then $L$ is semisimple. In particular, a g-radical
supplemented module $M$ with $Rad_{g}M=0$ is semisimple.
\end{lemma}

\begin{proof}
Let $X\leq L$. Since $M$ is g-radical supplemented, there exists a g-radical
supplement $T$ of $X$ in $M$. Hence $M=X+T$ and $X\cap T\leq Rad_{g}T\leq
Rad_{g}M$. Since $M=X+T$ and $X\leq L$, by Modular Law, $L=L\cap M=L\cap
\left( X+T\right) =X+L\cap T$. Since $X\cap T\leq Rad_{g}M$ and $L\cap
Rad_{g}M=0$, $X\cap L\cap T=L\cap X\cap T\leq L\cap Rad_{g}M=0$. Hence $%
L=X\oplus L\cap T$ and $X$ is a direct summand of $L$.
\end{proof}

\begin{proposition}
Let $M$ be a g-radical supplemented module. Then $M=K\oplus L$ for some
semisimple module $K$ and some module $L$ with essential generalized radical.
\end{proposition}

\begin{proof}
Let $K$ be a complement of $Rad_{g}M$ in $M.$ Then by $\left[ \text{\ref{1.1}%
},17.6\right] $, $K\oplus Rad_{g}M\trianglelefteq M$. Since $K\cap
Rad_{g}M=0 $, then by Lemma \ref{9}, $K$ is semisimple. Since $M$ is
g-radical supplemented, there exists a g-radical supplement $L$ of $K$ in $M$%
. Hence $M=K+L$ and $K\cap L\leq Rad_{g}L\leq Rad_{g}M$. Then by $K\cap
Rad_{g}M=0$, $K\cap L=0$. Hence $M=K\oplus L.$ Since $M=K\oplus L$, then by
Lemma \ref{8}, $Rad_{g}M=Rad_{g}K\oplus Rad_{g}L$. Hence $K\oplus
Rad_{g}M=K\oplus Rad_{g}L$. Since $K\oplus Rad_{g}L=K\oplus
Rad_{g}M\trianglelefteq M=K\oplus L$, then by $\left[ \text{\ref{1.2},
Proposition 5.20}\right] $, $Rad_{g}L\trianglelefteq L$.
\end{proof}

\begin{proposition}
Let $M$ be an $R-$module and $U\leq M$. the following statements are
equivalent.

$\left( 1\right) $ There is a decomposition $M=X\oplus Y$ with $X\leq U$ and 
$U\cap Y\leq Rad_{g}Y$.

$\left( 2\right) $ There exists an idempotent $e\in End\left( M\right) $
with $e\left( M\right) \leq U$ and $\left( 1-e\right) \left( U\right) \leq
Rad_{g}\left( 1-e\right) \left( M\right) $.

$\left( 3\right) $ There exists a direct summand $X$ of $M$ with $X\leq U$
and $U/X\leq Rad_{g}\left( M/X\right) $.

$\left( 4\right) $ $U$ has a g-radical supplement $Y$ such that $U\cap Y$ is
a direct summand of $U$.
\end{proposition}

\begin{proof}
$\left( 1\right) \Rightarrow \left( 2\right) $ For a decomposition $%
M=X\oplus Y$, there exists an idempotent $e\in End\left( M\right) $ with $%
X=e\left( M\right) $ and $Y=\left( 1-e\right) \left( M\right) $. Since $%
e\left( M\right) =X\leq U$, we easily see that $\left( 1-e\right) \left(
U\right) =U\cap \left( 1-e\right) \left( M\right) $. Then by $Y=\left(
1-e\right) \left( M\right) $ and $U\cap Y\leq Rad_{g}Y$, $\left( 1-e\right)
\left( U\right) =U\cap \left( 1-e\right) \left( M\right) =U\cap Y\leq
Rad_{g}Y=Rad_{g}\left( 1-e\right) \left( M\right) $.

$\left( 2\right) \Rightarrow \left( 3\right) $ Let $X=e\left( M\right) $ and 
$Y=\left( 1-e\right) \left( M\right) $. Since $e\in End\left( M\right) $ is
idempotent, we easily see that $M=X\oplus Y$. Then $M=U+Y$. Since $e\left(
M\right) =X\leq U$, we easily see that $\left( 1-e\right) \left( U\right)
=U\cap \left( 1-e\right) \left( M\right) $. Since $M=U+Y$ and $U\cap Y=U\cap
\left( 1-e\right) \left( M\right) =\left( 1-e\right) \left( U\right) \leq
Rad_{g}\left( 1-e\right) \left( M\right) =Rad_{g}Y$, $Y$ is a g-radical
supplement of $U$ in $M$. Then by Lemma \ref{6}, $M/X=\left( Y+X\right) /X$
is a g-radical supplement of $U/X$ in $M/X$. Hence $U/X=\left( U/X\right)
\cap \left( M/X\right) \leq Rad_{g}\left( M/X\right) $.

$\left( 3\right) \Rightarrow \left( 4\right) $ Let $M=X\oplus Y$. Since $%
X\leq U$, $M=U+Y$. Let $t\in U\cap Y$ and $Rt+T=Y$ for an essential
submodule $T$ of $Y$. Let $\left( \left( T+X\right) /X\right) \cap \left(
L/X\right) =0$ for a submodule $L/X$ of $M/X$. Then $\left( L\cap T+X\right)
/X=\left( \left( T+X\right) /X\right) \cap \left( L/X\right) =0$ and $L\cap
T+X=X$. Hence $L\cap T\leq X$ and since $X\cap Y=0$, $L\cap T\cap Y\leq
X\cap Y=0$. Since $L\cap Y\cap T=L\cap T\cap Y=0$ and $T\trianglelefteq Y$, $%
L\cap Y=0$. Since $X\leq L$ and $M=X+Y$, by Modular Law, $L=L\cap M=L\cap
\left( X+Y\right) =X+L\cap Y=X+0=X$. Hence $L/X=0$ and $\left( T+X\right)
/X\trianglelefteq M/X$. Since $Rt+T=Y$, $R(t+X)+\left( T+X\right) /X=\left(
Rt+X\right) /X+\left( T+X\right) /X=\left( Rt+T+X\right) /X=\left(
Y+X\right) /X=M/X$. Since $t\in U$, $t+X\in U/X\leq Rad_{g}\left( M/X\right) 
$ and hence $R\left( t+X\right) \ll _{g}M/X$. Then by $R(t+X)+\left(
T+X\right) /X$ and $\left( T+X\right) /X\trianglelefteq M/X$, $\left(
T+X\right) /X=M/X$ and then $X+T=M$. Since $X+T=M$ and $T\leq Y$, by Modular
Law, $Y=Y\cap M=Y\cap \left( X+T\right) =X\cap Y+T=0+T=T$. Hence $Rt\ll
_{g}Y $ and by Lemma \ref{13}, $Y$ is a g-radical supplement of $U$ in $M$.
Since $M=X\oplus Y$ and $X\leq U$, by Modular Law, $U=U\cap M=U\cap \left(
X\oplus Y\right) =X\oplus U\cap Y$. Hence $U\cap Y$ is a direct summand of $%
U $.

$\left( 4\right) \Rightarrow \left( 1\right) $ Let $U=X\oplus U\cap Y$ for a
submodule $X$ of $U$. Since $Y$ is a g-radical supplement of $U$ in $M$, $%
M=U+Y$ and $U\cap Y\ll _{g}Y$. Hence $M=U+Y=\left( X\oplus U\cap Y\right)
+Y=X\oplus Y$.
\end{proof}

\begin{theorem}
Let $V$ be a g-radical supplement of $U$ in $M.$ If $U$ is a generalized
maximal submodule of $M$, then $U\cap V$ is a unique generalized maximal
submodule of $V$.
\end{theorem}

\begin{proof}
Since $U$ is a generalized maximal submodule of $M$ and $V/\left( U\cap
V\right) \simeq \left( V+U\right) /U=M/U$, $U\cap V$ is a generalized
maximal submodule of $V$. Hence $Rad_{g}V\leq U\cap V$ and since $U\cap
V\leq Rad_{g}V$, $Rad_{g}V=U\cap V$. Thus $U\cap V$ is a unique generalized
maximal submodule of $V$.
\end{proof}

\begin{definition}
Let $M$ be an $R-$module. If every proper essential submodule of $M$ is
generalized small in $M$ or $M$ has no proper essential submodules, then $M$
is called a generalized hollow module.
\end{definition}

Clearly we see that every hollow module is generalized hollow.

\begin{definition}
Let $M$ be an $R-$module. If $M$ has a large proper essential submodule
which contain all essential submodules of $M$ or $M$ has no proper essential
submodules, then $M$ is called a generalized local module.
\end{definition}

Clearly we see that every local module is generalized local.

\begin{proposition}
Generalized hollow and generalized local modules are g-supplemented, so are
g-radical supplemented.
\end{proposition}

\begin{proof}
Clear from definitions.
\end{proof}

\begin{proposition}
Let $M$ be an $R-$module and $Rad_{g}M\neq M$. Then $M$ is generalized
hollow if and only if $M$ is generalized local.
\end{proposition}

\begin{proof}
$\left( \Longrightarrow \right) $ Let $M$ be generalized hollow and let $L$
be a proper essential submodule of $M$. Then $L\ll _{g}M$ and by Lemma \ref%
{2}, $L\leq Rad_{g}M$. Thus $Rad_{g}M$ is a proper essential submodule of $M$
which contain all proper essential submodules of $M$.

$\left( \Longleftarrow \right) $ Let $M$ be a generalized local module, $T$
be a large essential submodule of $M$ and $L$ be a proper essential
submodule of $M$. Let $L+S=M$ with $S\trianglelefteq M$. If $S\neq M$, then $%
L+S\leq T\neq M.$ Thus $S=M$ and $L\ll _{g}M.$
\end{proof}

\begin{example}
\label{14}Consider the $%
\mathbb{Z}
-$module $%
\mathbb{Q}
.$ Since $Rad_{g}%
\mathbb{Q}
=Rad%
\mathbb{Q}
=%
\mathbb{Q}
$, $_{%
\mathbb{Z}
}%
\mathbb{Q}
$ is g-radical supplemented. But, since $_{%
\mathbb{Z}
}%
\mathbb{Q}
$ is not supplemented and every nonzero submodule of $_{%
\mathbb{Z}
}%
\mathbb{Q}
$ is essential in $_{%
\mathbb{Z}
}%
\mathbb{Q}
$, $_{%
\mathbb{Z}
}%
\mathbb{Q}
$ is not g-supplemented.
\end{example}

\begin{example}
\label{15}Consider the $%
\mathbb{Z}
-$module $%
\mathbb{Q}
\oplus 
\mathbb{Z}
_{p^{2}}$ for a prime $p$. It is easy to check that $Rad_{g}%
\mathbb{Z}
_{p^{2}}\neq 
\mathbb{Z}
_{p^{2}}$. By Lemma \ref{8}, $Rad_{g}\left( 
\mathbb{Q}
\oplus 
\mathbb{Z}
_{p^{2}}\right) =Rad_{g}%
\mathbb{Q}
\oplus Rad_{g}%
\mathbb{Z}
_{p^{2}}\neq 
\mathbb{Q}
\oplus 
\mathbb{Z}
_{p^{2}}$. Since $%
\mathbb{Q}
$ and $%
\mathbb{Z}
_{p^{2}}$ are g-radical supplemented, by Lemma \ref{5}, $%
\mathbb{Q}
\oplus 
\mathbb{Z}
_{p^{2}}$ is g-radical supplemented. But $%
\mathbb{Q}
\oplus 
\mathbb{Z}
_{p^{2}}$ is not g-supplemented.
\end{example}

$\bigskip $

\end{document}